\documentclass[a4paper,11pt]{article}

\usepackage{graphicx,psfrag}
\usepackage{amssymb,amsmath,amsthm}
\usepackage{enumerate}  
\usepackage{a4wide}
\usepackage{hyperref}

\usepackage{pstricks}
\usepackage{pst-pdf}
\usepackage{wrapfig}
\usepackage[font={small}]{caption}  

\usepackage{authblk}

\numberwithin{equation}{section}
\allowdisplaybreaks[4]

\newtheorem{proposition}{Proposition}[section]
\newtheorem{theorem}[proposition]{Theorem}
\newtheorem{lemma}[proposition]{Lemma}
\newtheorem{corollary}[proposition]{Corollary}
\theoremstyle{definition}
\newtheorem{definition}[proposition]{Definition}
\newtheorem{remark}[proposition]{Remark}

\renewenvironment{proof}{\smallskip\noindent\emph{\textbf{Proof.}}%
  \hspace{1pt}}{\hspace{-5pt}{\nobreak\quad\nobreak\hfill\nobreak%
    $\square$\vspace{2pt}\par}\smallskip\goodbreak}

\newenvironment{proofof}[1]{\smallskip\noindent{\textbf{Proof~of~#1.}}%
  \hspace{1pt}}{\hspace{-5pt}{\nobreak\quad\nobreak\hfill\nobreak%
    $\square$\vspace{2pt}\par}\smallskip\goodbreak}

\newcommand{\C}[1]{\mathbf{C}^{#1}}

\renewcommand{\L}[1]{{\mathbf{L}^#1}}
\newcommand{\Lloc}[1]{{\mathbf{L}_{loc}^{#1}}}

\newcommand{\W}[2]{{\mathbf{W}^{#1,#2}}}

\newcommand{\modulo}[1]{{\left|#1\right|}}

\newcommand{\reali}{{\mathbb{R}}}
\newcommand{\naturali}{{\mathbb{N}}}

\renewcommand{\epsilon}{\varepsilon}
\renewcommand{\phi}{\varphi}
\renewcommand{\theta}{\vartheta}

\newcommand{\spt}{\mathop{\rm spt}}
\newcommand{\sgn}{\mathop{\rm sgn}}

\newcommand{\diam}{\mathop{\rm diam}}
\renewcommand{\d}[1]{\mathinner{\mathrm{d}{#1}}}
\renewcommand{\div}{\mathop{\rm div}}

\newcommand{\co}{\mathop{\rm co}}

\newcommand{\Int}{\mathop{\rm int}}

\begin{document}

\title{Estimates of the Domain of Dependence for Scalar Conservation Laws}

\author[1, 2]{\small Nikolay Pogodaev}

\affil[1]{\footnotesize Krasovskii Institute of Mathematics and Mechanics \authorcr 
Kovalevskay str., 16, Yekaterinburg, 620990, Russia}
\affil[2]{\footnotesize Matrosov Institute for System Dynamics and Control
  Theory\authorcr Lermontov str., 134, Irkutsk, 664033, Russia}

\date{}


\maketitle

\begin{abstract}

  \noindent We consider the Cauchy problem for a multidimensional scalar
  conservation law and construct an outer estimate for the domain of dependence
  of its Kru\v{z}kov solution. The estimate can be represented as the
  controllability set of a specific differential inclusion. In addition,
  reachable sets of this inclusion provide outer estimates for the support of
  the wave profiles. Both results follow from a modified version of the
  classical Kru\v{z}kov uniqueness theorem, which we also present in the paper. Finally,
  the results are applied to a control problem consisting in steering a
  distributed quantity to a given set.

  \medskip

  \noindent\textit{2010~Mathematics Subject Classification: 35L65, 93B03}

  \medskip

\end{abstract}

\section{Introduction}

The paper aims at constructing an outer estimate for the domain of dependence of the Kru\v{z}kov solution to
the following Cauchy problem
\begin{gather}
\label{eq:genpde}
  \partial_t u + \div \left(f(t,x,u)\right) = 0,\\
\label{eq:intcon}
  u(0,x) = u_0(x), \qquad x\in\reali^n.
\end{gather}
Recall the precise definition of the domain of dependence.

\begin{definition}[D. Serre~\cite{Serre1}]
\label{def:dd}
Let $u$ be a Kru\v{z}kov solution to~\eqref{eq:genpde},~\eqref{eq:intcon}.
The \emph{domain of dependence} of $u$ at a point $(t,x)$ is the smallest compact set $\mathcal{D}_u(t,x)\subset\reali^n$ such that, for every bounded function
$w$ with compact support disjoint from $\mathcal{D}_u(t,x)$ and every positive $\epsilon$
small enough, the solution of the Cauchy problem with the
initial condition $v_0 =u_0 + \epsilon w$ coincides with $u$ at $(t,x)$.
\end{definition}

Two things concerning this definition must be clarified. First of all, since initial functions $u_0$ usually belong to $\L\infty(\reali^n)$, by the support of $u_0$ we mean
the support of the measure $u_0\mathcal{L}^n$, i.e., the smallest closed set $A$ such that
\[
\int_{\reali^n\setminus A} u_0(x)\d x = 0.
\]
Throughout the paper, this set is denoted by $\spt u_0$.

Another issue follows from the fact that a Kru\v{z}kov solution is defined as a
map of class 
$\C0\left([0,\infty);\Lloc1(\reali^n)\right)$. By saying that two solutions $u$ and $v$ coincide at $(t,x)$,
we mean that
\begin{equation}\label{eq:conicide}
\lim_{r\to 0}\frac{1}{\mathcal{L}^n\left(B(x,r)\right)}\int_{B(x,r)}|u(t,y)-v(t,y)|\d y = 0,
\end{equation}
where $B(x,r)$ denotes the closed ball of radius $r$ centered at $x$.

For a very special case of~\eqref{eq:genpde}, when the equation is one-dimensional and the flow is convex in $u$,
the domain of dependence can be found explicitly by the method of generalized characteristics~\cite{Dafermos}. In the 
general case, a rough outer estimate is provided by the Kru\v{z}kov uniqueness theorem~\cite{Kruzkov}:
\[
\mathcal{D}_u(t,x)\subseteq B(x,ct),
\]
where $c=\sup\left\{|\partial_u f(s,x,u)|\;\colon\; s\in [0,t],\; x\in \reali^n,\; u\in \reali\right\}$.

Note that the ball $B(x,ct)$ is exactly \emph{the controllability set}
of the differential inclusion
\begin{equation}
\label{eq:simpledi}
\dot y(s)\in B(0,c),
\end{equation}
i.e., the set of all points $a\in \reali^n$ that can be connected with $x$ by a 
trajectory $y\colon [0,t]\to\reali^n$ of~\eqref{eq:simpledi}. 

Encouraged by this observation, we are going to find a differential inclusion
\begin{equation}
\label{eq:maindi}
\dot y(s)\in F\left(s,y(s)\right),
\end{equation}
whose right-hand side is smaller than $B(0,c)$ and
whose controllability set still gives an outer estimate of $\mathcal{D}_u(t,x)$. As we shall see, 
a possible choice for such set-valued map $F$ is
\[
  F(s,y) = \co\partial_uf\left(s,y,\left[a(s),b(s)\right]\right),
\]
where $a$ and $b$ are certain upper and lower bounds of $u$, while `$\co$' denotes the convex hull of a set.

To check that the controllability set contains the domain of dependence, we use a modified version of the classical Kru\v{z}kov theorem. In this version the cone 
appearing in the original theorem is substituted by the backward integral funnel of~\eqref{eq:maindi}. A difficult moment
appears at this stage: our arguments work only when the integral funnel is sufficiently regular. 
Hence, before proving the theorem, we spend some time discussing the corresponding regularity issues. 


The paper is organized as follows. In Section 2 we recall the necessary
information concerning integral funnels and their regularity.
A modified version of the Kru\v{z}kov theorem is established in
Section 3. We exploit this theorem in Section 4 to obtain estimates for the
domain of dependence and the support of the wave profile $u(t,\cdot)$. Next, we
use the latter result 
for analysing a control problem consisting in steering a distributed
quantity  to a given set. For the ease of presentation, the proofs of various
technical lemmas and propositions, devoted solely to the properties of
differential inclusions, are collected in Appendix.

\section{Integral funnels and their regularity}

In this section we discuss the notion of integral funnel and its essential properties. 
We begin with a short list of notations which are used throughout the paper.

\paragraph{Notation.} 
In what follows, $|x|$ is the Euclidean norm and $x\cdot y$ is
the scalar product of $x,y\in\reali^n$. Given a closed set $A\subseteq\reali^n$, we denote by $d_A(x)$ the distance
between $x\in\reali^n$ and $A$, i.e., $d_A(x) = \inf_{a\in A}|x-a|$, and by $B(A,r)$ the closed $r$-neighbourhood of $A$,
i.e., $B(A,r) = \{x\;\colon\; d_A(x)\leq r\}$. Finally, given an arbitrary set $A$, we use the symbol $\co A$ for its closed convex hull,
$A^c$ for its complement, $\partial A$ for its topological boundary, and
$\mathcal{H}^n(A)$ for its Hausdorff measure.

\vspace{10pt}

Let us consider a set-valued map $F=F(t,x)$ defined by the rule
\begin{equation}
\label{eq:F}
F(t,x) = \co g(t,x,U),
\end{equation}
where $U\subset\reali^l$ is compact and $g\colon[0,\infty)\times\reali^n\times U\to\reali^n$ satisfies the 
following assumptions:
\begin{enumerate}
\item[$(\mathbf{g_{1}})$] $g$ is continuous;
\item[$(\mathbf{g_{2}})$] for each $u$ the map $g(\cdot,\cdot,u)$ is Lipschitz
with modulus $L_1$:
\[
|g(t,x,u)-g(t',x',u)|\leq L_1\left(|t-t'|+|x-x'|\right)
\;\text{for all }\; t,t'\in [0,\infty),\;x,x'\in \mathbb{R}^n,\;u\in U;
\]
\item[$(\mathbf{g_{3}})$] for each $t$ and $u$ the map $g(t,\cdot,u)$ is continuously differentiable and its
derivative $\frac{\partial g}{\partial x}(t,\cdot, u)$ is Lipschitz with modulus $L_2$:
\[
\left|\frac{\partial g}{\partial x} (t,x,u)-\frac{\partial g}{\partial x} (t,x',u)\right|\leq L_2|x-x'|
\;\;\text{for all }\; t\in [0,\infty),\;x,x'\in \mathbb{R}^n,\; u\in U.
\]
\end{enumerate}

Fix a compact interval $[\tau_0,\tau]\subset [0,\infty)$ and consider the following differential inclusion
\begin{equation}
\label{eq:inclusion}
\dot x(t) \in F\left(t,x(t)\right)\qquad\mbox{for a.e.}\;\; t\in [\tau_0,\tau].
\end{equation}

The notion of integral funnel generalizes, in some sense, the usual notion of trajectory.

\begin{definition}
Let $K\subset \reali^n$ be compact. The set
\[
\Omega^+(K) = 
\left\{\left(t,x(t)\right)\;\colon\;t\in [\tau_0,\tau],\;x(\cdot)\;
\mbox{is a solution to}~\eqref{eq:inclusion},
\;x(\tau_0)\in K\right\}.
\]
is called  the \emph{forward integral funnel} issuing from $K$.
\end{definition}

Under assumptions $(\mathbf{g_{1}})$--$(\mathbf{g_{3}})$ the integral funnel is
a nonempty compact subset of $\reali^{n+1}$. Moreover, each ``slice''
\[
\Omega^+_t(K) = \{x\;\colon\;(t,x)\in \Omega^+(K)\}, \quad t\in [\tau_0,\tau],
\]
called the \emph{reachable set} of~\eqref{eq:inclusion} at time $t$, is nonempty and compact in $\reali^n$ (see, e.g.,~\cite{Tolstonogov}).

Besides compactness, the funnel $\Omega^+(K)$ has certain regularity properties. To be more precise,
we need the following extra definitions.

\begin{definition}
\label{def:tubular}
A set $E \subseteq\reali^n$ is called a \emph{tubular neighborhood}\footnote{The synonyms are: sets with interior sphere property (P. Cannarsa, H. Frankowska), sets with interior
ball property (O. Alvarez, P. Cardaliaguet, R. Monneau), sets with positive erosion (T. Lorenz), parallel sets (L. Ambrosio,
A. Colesanti, E. Villa).} (of a subset of $\reali^n$) if $E = B(A,r)$ for a
closed set $A\subset E$ and a positive $r$.
\end{definition}

\begin{definition}
\label{def:rect}
A set $E\subset\reali^n$ is said to be \emph{$m$-rectifiable} if there exists a Lipschitz function $f$ mapping a
\emph{bounded} subset $A \subset \reali^m$ onto $E$\footnote{In particular, the definition implies that all $m$-rectifiable sets have finite $m$-dimensional
Hausdorff measure.}.
\end{definition}

\begin{definition}
\label{def:content}
Let $E\subseteq\reali^n$. The limit 
$\mathcal{SM}^n(E)=\lim_{r\to 0+}\frac{1}{r}\mathcal{H}^n\left(B(E,r)\setminus E\right)$,
when it exists, is called the \emph{outer Minkowski content} of $E$.
\end{definition}

As the next proposition demonstrates, tubular neighbourhoods have nice regularity properties.

\begin{proposition}
\label{prop:tubular}
Any compact $n$-dimensional tubular neighbourhood $A$ has $(n-1)$-rectifiable topological boundary $\partial A$ and admits
the finite outer Minkowski content.
\end{proposition}

If the initial set $K$ is an $n$-dimensional tubular neighbourhood, the funnel $\Omega^+(K)$ looks almost like an $(n+1)$-dimensional
tubular neighbourhood (see Appendix for details). In particular, it has similar regularity properties.

\begin{proposition}
\label{prop:rectifiable}
Let $K\subset\reali^n$ be a compact tubular neighbourhood. Then
\begin{enumerate}[\textnormal{(}a\textnormal{)}]
	\item $\Omega^+(K)$ has $n$-rectifiable boundary and $\mathcal{SM}^{n+1}\left(\Omega^+(K)\right)=\mathcal{H}^n\left(\partial\Omega^+(K)\right)<\infty$;
	\item $\Omega_t^+(K)$ is an $n$-dimensional tubular neighbourhood, for each $t\in [\tau_0,\tau]$.
\end{enumerate} 
\end{proposition}
It is worth to mention that slices of an arbitrary tubular neighbourhood are not necessarily tubular neighbourhoods, while the
slices of the funnel are. The proofs of Propositions~\ref{prop:tubular},~\ref{prop:rectifiable} are defered to Appendix.

Integral funnels may also be characterized in terms of proximal
normals.

\begin{definition}
\label{def:prox}
A vector $p\in\reali^n$ is a \emph{proximal normal} to a closed set $C\subset \mathbb{R}^n$ 
at a point $x\in C$ if there is $y\not\in C$ such that $|y-x| = d_C(y)$ and $p = \alpha(y - x)$ 
for some $\alpha> 0$. The set of all proximal normals to $C$ at $x$ is a cone denoted by $N^P_C(x)$.
\end{definition}

To state the next result, consider the map $H\colon[0,\infty)\times\reali^n\times\reali^n\to\reali$ defined by the rule
\[
H(t,x,p) = \max\left\{p\cdot v\;\colon\; v\in F(t,x)\right\}
\]
and called the \emph{(upper) Hamiltonian} associated to $F$.

\begin{proposition}[F. Clarke~\cite{ClarkeProx}]
\label{prop:proximal}
Let $K\subset\reali^n$ be compact. Then,
for every $(t,x)\in\Omega^+(K)$ with $\tau_0<t<\tau$, we have
\begin{equation}
\label{eq:hamilton}
\theta + H(t,x,\zeta) = 0
\qquad\text{for all}\qquad (\theta,\zeta)\in N_{\Omega^+(K)}^P(t,x).
\end{equation}
\end{proposition}

Actually, under our assumptions, there are only two possibilities: the cone $N_{\Omega^+(K)}^P(t,x)$ is either zero or consists of a single ray. This fact
follows from Proposition~\ref{prop:rectifiable}.

\begin{remark}
\label{rem:backward}
Together with the forward integral funnel $\Omega^+(K)$ one may consider the 
\emph{backward integral funnel}:
\[
\Omega^-(K) = 
\left\{\left(t,x(t)\right)\;\colon\;t\in [\tau_0,\tau],\;x(\cdot)\;
\mbox{is a solution to}~\eqref{eq:inclusion},
\;x(\tau)\in K\right\}.
\]
It is easy to see that $x=x(t)$ satisfies~\eqref{eq:inclusion} if and only if $y=y(t)$,
defined by $y(t) = x(\tau+\tau_0-t)$, satisfies
\begin{equation}
\label{eq:backward_inclusion}
\dot y(t) \in \hat F\left(t,y(t)\right)\qquad \mbox{for a.e.}\;\; t\in [\tau_0,\tau],
\end{equation}
where
\[
\hat F(t,y) = -F\left(\tau+\tau_0-t,y\right).
\]
Moreover, $x(\tau)\in K$ is equivalent to $y(\tau_0)\in K$. Thus, denoting the forward funnel
of~\eqref{eq:backward_inclusion} by $\hat\Omega^+(K)$, we obtain
\[
\Omega^-(K) = \left\{(t,x)\;\colon\;t\in [\tau_0,\tau],\;\; (\tau+\tau_0-t,x)\in \hat{\Omega}^+(K)\right\}.
\]

The above identity implies that Proposition~\ref{prop:rectifiable} holds also for the backward funnel of~\eqref{eq:inclusion}. In order to rewrite 
Proposition~\ref{prop:proximal}, 
we notice that $(\theta,\zeta)$ is a proximal normal to $\Omega^-(K)$ at $(t,x)$ if and only if 
$(-\theta,\zeta)$ is a proximal normal to $\hat\Omega^+(K)$ at $(\tau+\tau_0-t,x)$. Moreover,
\begin{align*}
\hat H(\tau+\tau_0-t,x,p) 
&= \max\{\langle p,v\rangle\;\colon\;v\in \hat F(\tau+\tau_0-t,x)\}\\
&= \max\{\langle p,v\rangle\;\colon\;v\in -F(t,x)\}\\
&= H(t,x,-p).
\end{align*}
Therefore, equation~\eqref{eq:hamilton} must be substituted with
\begin{equation}
\label{eq:back_hamilton}
-\theta + H(t,x,-\zeta) = 0
\qquad\text{for all}\qquad (\theta,\zeta)\in N_{\Omega^-(K)}^P(t,x).
\end{equation}

\end{remark}

\section{The modified Kru\v{z}kov theorem}

In this section we prove the key result of our paper, a certain modification of
the classical Kru\v{z}kov theorem.

\paragraph{Assumptions.} Throughout the rest of the paper, the map $f\colon [0,\infty)\times\reali^n\times \reali\to\reali^n$ has the following properties:
\begin{enumerate}
\item[$(\mathbf{f_{1}})$] $f$ is continuously differentiable and its partial derivatives $u\mapsto\frac{\partial f}{\partial t}(t,x,u)$ and $u\mapsto\frac{\partial f}{\partial x}(t,x,u)$
are Lipschitz for all $t$ and $x$;
\item[$(\mathbf{f_{2}})$] $\frac{\partial f}{\partial u}$ satisfies assumptions $(\mathbf{g_{1}})$--$(\mathbf{g_{3}})$.
\end{enumerate}

\begin{definition}
A \emph{Kru\v{z}kov solution} of~\eqref{eq:genpde},~\eqref{eq:intcon} with $u_0 \in \L\infty(\reali^n)$ is a bounded measurable function
$u\colon [0,\infty)\times\reali^n\to \reali$ such that
$u\in \C0\left([0,\infty); \Lloc1(\reali^n)\right)$, $u(0,\cdot) = u_0$, and
\begin{multline}
\label{eq:admissible}
\iint
\bigg[\modulo{u(t,x)-k}\phi_t
 + \sgn\left(u(t,x)-k\right)
\cdot
\\
\sum_{\alpha=1}^m
\left(
\left[
f_{\alpha}\left(t,x,u(t,x)\right)-
f_{\alpha}(t,x,k)
\right]\phi_{x_{\alpha}}
-\partial_{x_{\alpha}}f_{\alpha}(t,x,k)\,\phi
\right)
\bigg]
\d t\d x 
\geq 0,
\end{multline}
for each $k\in \reali$ and every nonnegative Lipschitz test function $\phi=\phi(t,x)$
with compact support contained in the half-space $t>0$.
\end{definition}

Here, following~\cite{Dafermos}, we choose Lipschitz test functions instead
of smooth ones. The above definition is equivalent to the classical one given by~Kru\v{z}kov in~\cite{Kruzkov}. 
To see this, one may apply the same arguments as in the lemma below.

\begin{lemma}
\label{lem:sum}
Let $u$ and $\bar u$ be Kru\v{z}kov solutions of~\eqref{eq:genpde},~\eqref{eq:intcon} with initial data $u_0$ and $\bar u_0$, respectively.
Then
\begin{multline}
\label{eq:sum_v}
\iint
\bigg[
\left|u(t,x)-\bar u(t,x)\right|
\phi_t(t,x)
+\sgn\left(u(t,x)-\bar u(t,x)\right)\cdot\\
\sum_{\alpha=1}^n
\left[f_{\alpha}\left(t,x,u(t,x)\right)-f_{\alpha}\left(t,x,\bar u(t,x)\right)\right]
\phi_{x_{\alpha}}(t,x)
\bigg]
\d{t}\d{x}
\geq 0,
\end{multline}
for every nonnegative Lipschitz test function $\phi=\phi(t,x)$
with compact support contained in the half-space $t>0$.
\end{lemma}
\begin{proof}
Kru\v{z}kov proved that~\eqref{eq:sum_v} holds when $\phi$ is \emph{smooth}
(see~\cite[formula (3.12)]{Kruzkov}). Now, assuming that $\phi$ is Lipschitz, we
approximate it by smooth functions.

Let $\eta_{\epsilon}$ be the standard mollifier. Then $\phi_{\epsilon}=\eta_{\epsilon}*\phi$ is a smooth nonnegative
function compactly supported in the half-space $t>0$ for every small $\epsilon$. Moreover, by~\cite[Theorem 4.1]{EvansGariepy},
\begin{equation}
\label{eq:lemsum}
\partial_t\phi_{\epsilon}\to \partial_t\phi \quad\text{and}\quad \partial_{x_\alpha}\phi_{\epsilon}\to \partial_{x_\alpha}\phi
\quad\text{in}\quad\L1(\reali^{n+1}),
\end{equation}
where $\partial_t\phi$ and $\partial_{x_\alpha}\phi$ are the \emph{weak} partial derivatives of $\phi$. 

Any compactly supported Lipschitz function $\phi$ belongs to $\W1p(\reali^{n+1})$ for some $n+1<p<\infty$. Hence
the weak partial derivatives ($\partial_t\phi$ and $\partial_{x_\alpha}\phi$)
coincide with the classical partial derivatives
($\phi_t$ and $\phi_{x_\alpha}$) almost everywhere on $\reali^{n+1}$~\cite[Corollary~11.36]{Leoni}.  
The lemma now follows from the Kru\v{z}kov's result and~\eqref{eq:lemsum}.
\end{proof}



\begin{theorem}
\label{thm:main} 
Suppose that $(\mathbf{f_{1}})$, $(\mathbf{f_{2}})$ hold.
Let $u$ and $\bar u$ be Kru\v{z}kov solutions of~\eqref{eq:genpde},~\eqref{eq:intcon} with initial values $u_0$ and $\bar u_0$, respectively.
Then, for every compact tubular neighbourhood $K\subset\reali^n$ and all $\tau_0,\tau\in [0,\infty)$ such that $\tau_0\leq\tau$, we have
\begin{equation}
\label{eq:main}
\int_{K}\modulo{u(\tau,x)-\bar u(\tau,x)}\d x
\leq
\int_{\Omega^-_{\tau_0}(K)}\modulo{u(\tau_0,x)-\bar u(\tau_0,x)}\d x.
\end{equation}
Here $\Omega^-(K)$ is the backward integral funnel of differential inclusion~\eqref{eq:inclusion} with the right-hand side given by
\begin{equation}
\label{eq:F2}
F(t,x) = \co \partial_u f\left(t,x,[a,b]\right),
\end{equation}
where $a$ and $b$ are lower and upper bounds for solutions on $[\tau _0,\tau]$:
\[
    a \leq u(t,x) \leq b, \quad a \leq \bar u(t,x) \leq b \quad\text{for almost all}\quad 
    t\in [\tau_0,\tau],\;\; x\in\reali^n,
\]
\end{theorem}

\begin{proof}
To simplify the notation, we write $\Omega$ and $\Omega_t$ instead of $\Omega^-(K)$ and $\Omega^{-}_t(K)$.

\textbf{1.} First, we must construct an appropriate Lipschitz approximation of the characteristic function
$\mathbf{1}_{\Omega}$. Let us take the following one:
\[
\phi(t,x)=
\begin{cases}
1, & d_{\Omega}(t,x)=0,\\
1-\frac{1}{\epsilon}d_{\Omega}(t,x),& d_{\Omega}(t,x)< \epsilon,\\
0 & d_{\Omega}(t,x)\geq \epsilon,
\end{cases}
\]
where $d_{\Omega}(t,x)$ is the distance from $(t,x)$ to $\Omega$. 

\textbf{2.} Now we split the funnel's boundary $\partial\Omega$ into 3 parts:
\[
(\partial\Omega)^{-} = \partial\Omega\cap\{t=\tau_0\},\quad
(\partial\Omega)^{+} = \partial\Omega\cap\{t=\tau\},\quad
(\partial\Omega)^{s} = \partial\Omega\cap\{\tau_0<t<\tau\},
\]
then consider the set
\[
\Xi_{\epsilon} = \left\{(t,x)\;\colon\;0<d_{\Omega}(t,x)<\epsilon\right\}
\]
and split it into 5 parts:
\begin{align*}
\Xi_{\epsilon}^{\tau_0} &= \Xi_{\epsilon}\cap B\left((\tau_0,\partial\Omega_{\tau_0}),\epsilon\right),\\
\Xi_{\epsilon}^{\tau} &= \Xi_{\epsilon}\cap B\left((\tau,\partial\Omega_{\tau}),\epsilon\right),\\
\Xi_{\epsilon}^- &= \left\{(t,x)\in \Xi_{\epsilon}\;\colon\;\pi_\Omega(t,x)\subset (\partial\Omega)^-\right\}\setminus \Xi_{\epsilon}^{\tau_0},\\
\Xi_{\epsilon}^+ &= \left\{(t,x)\in \Xi_{\epsilon}\;\colon\;\pi_\Omega(t,x)\subset (\partial\Omega)^+\right\}\setminus \Xi_{\epsilon}^{\tau},\\
\Xi_{\epsilon}^s &= \left\{(t,x)\in \Xi_{\epsilon}\;\colon\;\pi_\Omega(t,x)\subset (\partial\Omega)^s\right\}\setminus (\Xi_{\epsilon}^{\tau_0}\cap \Xi_{\epsilon}^{\tau}).
\end{align*}
Here $\pi_{\Omega}(t,x)$ denotes the set of points $(s,y)$ satisfying 
$d_{\Omega}(t,x)=\sqrt{(t-s)^2+|x-y|^2}$; in other words, each point 
$(s,y)\in \pi_{\Omega}(t,x)$ is a projection of $(t,x)$ onto $\Omega$.

\begin{figure}[!ht]%
  \centering
\centering
\includegraphics[width=0.47\textwidth]{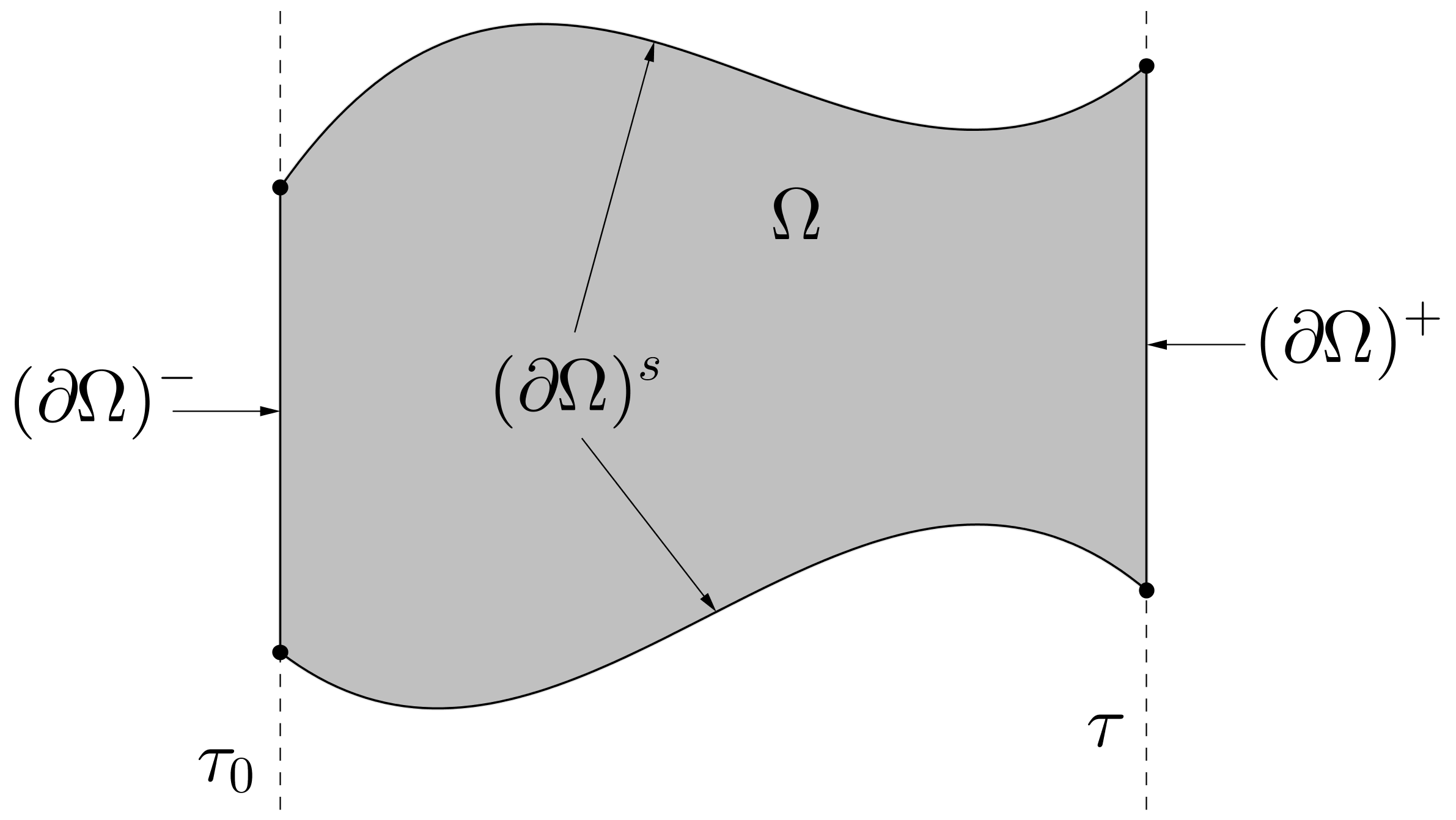}\quad
\includegraphics[width=0.47\textwidth]{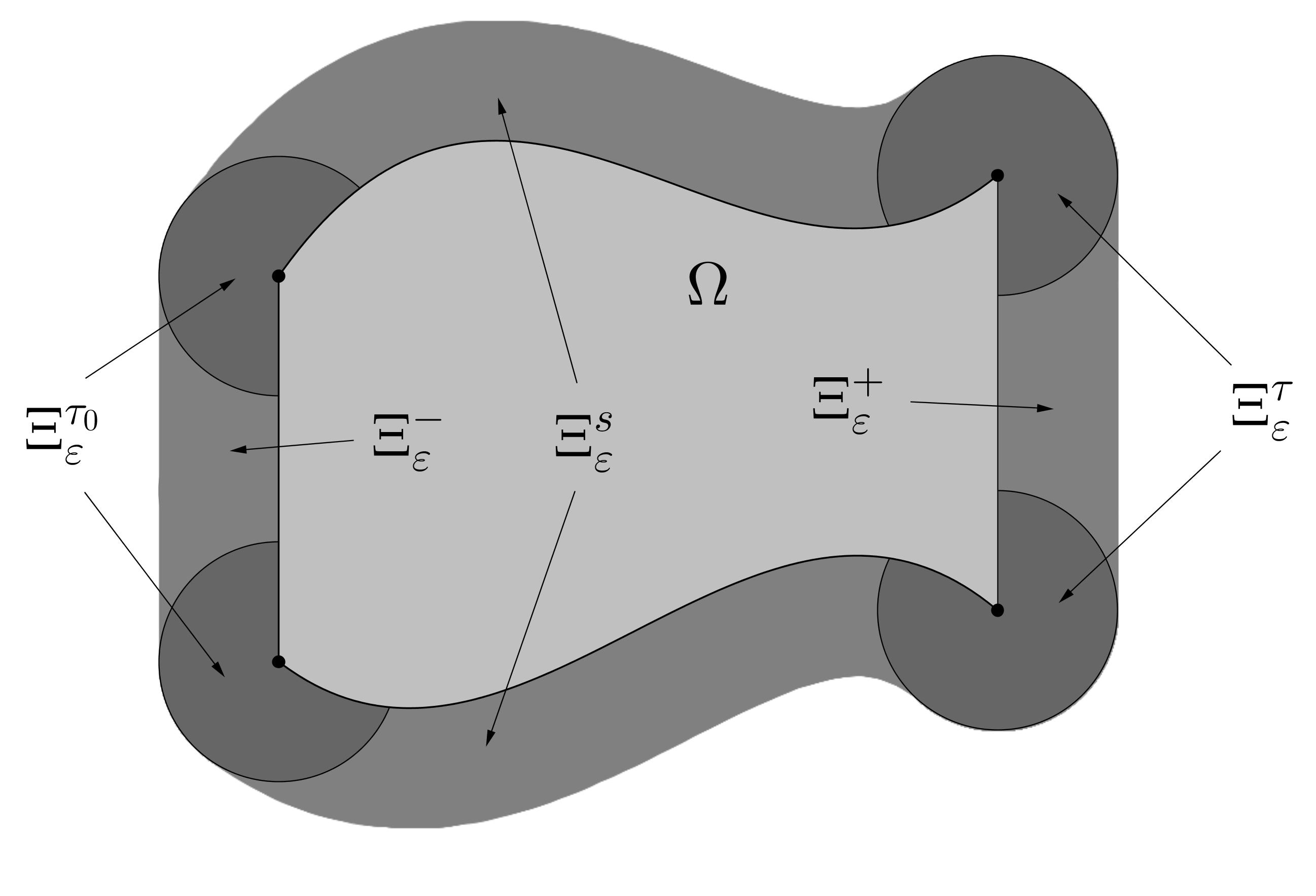}
\caption{\label{fig:funnel} The set $\Omega +rB$.}
\end{figure}

We state that, for all sufficiently small $\epsilon$, the following implications hold:
\begin{equation}
\label{eq:projections}
\begin{array}{l}
(t,x)\in \Xi_{\epsilon}^-\quad\Rightarrow\quad t<\tau_0\;\;\mbox{and}\;\;\pi_{\Omega}(t,x)=(\tau_0,x),\\
(t,x)\in \Xi_{\epsilon}^+\quad\Rightarrow\quad t>\tau\;\;\mbox{and}\;\;\pi_{\Omega}(t,x)=(\tau,x).
\end{array}
\end{equation} 
It is enough to check the first implication. 
Let $(t,x)\in \Xi_{\epsilon}^-$. Since $(t,x)\not\in\Xi_{\epsilon}^{\tau_0}$,
we conclude that $x\in\Int\Omega_{\tau_0}$ and $\pi_{\Omega}(t,x) = (\tau_0, x)$. Thus, the open $(n+1)$-dimensional ball of radius 
$(t- \tau_0)$ around $(t,x)$ does not contain points of $\Omega$. More specifically, if $y(\cdot)$ satisfies~\eqref{eq:maindi} and $y(\tau_0)=x$,
then, for all $s\in[\tau_0, \tau]$, it must be
\[
(t-s)^2 + |x-y(s)|^2 \geq (t -\tau_0)^2.
\]
Suppose that $t>\tau_0$. It follows from $(\mathbf{f_{1}})$ and $(\mathbf{f_{2}})$ that $y(\cdot)$ is Lipschitz. Hence, for some $C>0$ and all $s$, we have
\[
C^2(s - \tau_0)^2\geq |y(\tau_0)-y(s)|^2 = |x-y(s)|^2 \geq (t -\tau_0)^2 - (t-s)^2.
\]
After performing an easy computation, we obtain
\[
C^2\geq 1 + 2\frac{t-s}{s- \tau_0}.
\]
But the fraction on the right-hand side tends to $+\infty$ as $s\to \tau_0$. This gives a contradiction.




\textbf{3.} Assume that $\tau_0>0$. Then, for all sufficiently small $\epsilon$, the support of $\phi$
is contained in the half-plane $t>0$.
Insert $\phi$ into~\eqref{eq:sum_v}. By Propositions~\ref{prop:tubular},~\ref{prop:rectifiable} and Remark~\ref{rem:backward}, the boundary of $\Omega$ is regular enough
for assuring the equality
\[
\mathcal{H}^{n+1}(\partial\, \Xi_{\epsilon})=0.
\]
Hence, after we put $\phi$ into~\eqref{eq:sum_v}, the boundary points of $\Xi_{\epsilon}$ can be neglected. For all interior points of the sets 
$\Omega$ and $\left(\Xi_{\epsilon}\cup \Omega\right)^c$, we have $\nabla \phi=0$.
Thus, only the integral over $\Xi_{\epsilon}$ remains in~\eqref{eq:sum_v}.

\textbf{4.} Fix a point $(t,x)\in \Xi_{\epsilon}$, where $d_\Omega$ is differentiable.
In this case $(t,x)$ has a unique projection $(\bar t,\bar x)$ on $\Omega$ 
and $\nabla d_{\Omega}(t,x)=(\theta,\zeta)$ is a unit proximal normal to $\Omega$ at $(\bar t,\bar x)$.
Now the integrand in~\eqref{eq:sum_v} takes the form $-\frac{1}{\epsilon}g(t,x)$, where
\begin{equation*}
g(t,x) = 
\left|u(t,x)-\bar u(t,x)\right|
\theta
+
\sgn\left(u(t,x)-\bar u(t,x)\right)
\left[f\left(t,x,u(t,x)\right)-f\left(t,x,\bar u(t,x)\right)\right]\cdot
\zeta.
\end{equation*}
Using the obvious identity
\[
f\left(t,x,u(t,x)\right)-f\left(t,x,\bar u(t,x)\right) \\= 
\left(u(t,x)-\bar u(t,x)\right)\int_0^1\partial_u f\left(t,x,su(t,x)+(1-s)\bar u(t,x)\right)\d s,
\]
we get
\[
g(t,x) = 
\left|u(t,x)-\bar u(t,x)\right|
\left(
\theta + \int_0^1\partial_u f\left(t,x,su(t,x)+(1-s)\bar u(t,x)\right)\cdot\zeta\d s
\right).
\]
It follows from the Lipschitz continuity of $\partial_uf$ that
\[
\partial_uf\left(t,x,w\right)\cdot\zeta
\geq
\partial_uf\left(\bar t,\bar x,w\right)\cdot\zeta
-2L_1 \epsilon |\zeta|,\qquad w\in [a,b].
\]
Note that
\[
H(\bar t,\bar x,\zeta) 
= \max\left\{v\cdot \zeta\;\colon\; v\in F(\bar t,\bar x) \right\} 
= \max\left\{\partial_u f(\bar t,\bar x,w)\cdot \zeta\;\colon\; w\in [a,b] \right\}.
\]
Therefore, 
\[
\partial_uf\left(\bar t,\bar x,w\right)\cdot\zeta
\geq - H(\bar t, \bar x,-\zeta), \qquad  w\in [a,b],
\]
implying
\begin{equation}
\label{eq:important}
g(t,x)\geq \left|u(t,x)-\bar u(t,x)\right|\left(\theta-H(\bar t,\bar x,-\zeta)-2L_1 \epsilon |\zeta|\right).
\end{equation}

\textbf{5.} Let $(t,x)\in \Xi_{\epsilon}^s$. 
By the definition of $\Xi_{\epsilon}^s$, we have $(\bar t,\bar x)\in (\partial\Omega)^s$, so~\eqref{eq:back_hamilton} yields 
\[
\theta-H(\bar t,\bar x,-\zeta)=0
\qquad\text{for all}\qquad (\theta,\zeta)\in N_{\Omega^-(K)}^P(\bar t,\bar x).
\]
Now, it follows from~\eqref{eq:important} that
\[
-\iint_{\Xi_{\epsilon}^s} g(t,x)\d t\d x \leq 2L_1 \epsilon (b-a)\mathcal{H}^{n+1}(\Xi_{\epsilon}),
\]
because all non-differentiability points $(t,x)$ of $d_\Omega$ can be neglected by Rademacher's theorem.

If $(t,x)\in \Xi_{\epsilon}^+$, we deduce from~\eqref{eq:projections} that
$\theta=1$, $\zeta = 0$. Therefore,~\eqref{eq:important} gives
\[
-\iint_{\Xi_{\epsilon}^+} g(t,x)\d t \d x \leq -\iint_{\Xi_{\epsilon}^+} \modulo{u(t,x)-\bar u(t,x)}\d t \d x.
\]

Similarly, for $(t,x)\in \Xi_{\epsilon}^-$, we have
$\theta=-1$, $\zeta = 0$, which implies 
\[
-\iint_{\Xi_{\epsilon}^-} g(t,x)\d t \d x \leq \iint_{\Xi_{\epsilon}^-} \modulo{u(t,x)-\bar u(t,x)}\d t \d x.
\]

\textbf{6.} As for the rest part of $\Xi_{\epsilon}$, i.e.,
$\Xi_{\epsilon}^{\tau_0}\cup\Xi_{\epsilon}^{\tau}$, let us note that
it is contained in the set $B(E_{\tau_0},\epsilon)\cup B(E_{\tau},\epsilon)$, where
$E_{\tau_0} = \{\tau_0\}\times\partial \Omega_{\tau_0}$ and $E_{\tau} = \{\tau\}\times\partial \Omega_{\tau}$.
Therefore, 
\[
-\iint_{\Xi_{\epsilon}^{\tau_0}\cup\Xi_{\epsilon}^{\tau}} g(t,x)\d t\d x \leq 
M\left[
\mathcal{H}^{n+1}\left(B(E_{\tau_0},\epsilon)\right)
+
\mathcal{H}^{n+1}\left(B(E_{\tau},\epsilon)\right)
\right]
\]
for some $M>0$. Using Proposition~\ref{prop:rectifiable}(a) and Remark~\ref{rem:backward}, we conclude that
\[
\mathcal{H}^{n+1}\left(\Xi_{\epsilon}\right)=O(\epsilon).
\]
By Proposition~\ref{prop:rectifiable}(b) and Proposition~\ref{prop:tubular}, the sets $E_{\tau_0}$ and $E_\tau$ are $(n-1)$-rectifiable,
hence they have zero $(n+1)$-dimensional outer Minkowski content (see Lemma~\ref{lem:content}). In other words,	 
\[
\mathcal{H}^{n+1}\left(B(E_{\tau_0},\epsilon)\right)= o(\epsilon),
\quad
\mathcal{H}^{n+1}\left(B(E_{\tau},\epsilon)\right)=o(\epsilon).
\]

Combining all the previous inequalities, we obtain
\[
0\leq \frac{1}{\epsilon}\iint_{\Xi_{\epsilon}^-} \modulo{u(t,x)-\bar u(t,x)}\d t \d x
-\frac{1}{\epsilon}\iint_{\Xi_{\epsilon}^+} \modulo{u(t,x)-\bar u(t,x)}\d t \d x
+O(\epsilon).
\]
Passing to the limit as $\epsilon\to 0$ gives
\[
\lim_{\epsilon\to 0}\frac{1}{\epsilon}\int_{\tau}^{\tau+\epsilon} \int_{\Omega_\tau}\modulo{u(t,x)-\bar u(t,x)}\d x\d t
\leq
\lim_{\epsilon\to 0}\frac{1}{\epsilon}\int^{\tau_0}_{\tau_0-\epsilon} \int_{\Omega_{\tau_0}}\modulo{u(t,x)-\bar u(t,x)}\d x\d t.
\]
Using the fact that both $u$ and $\bar u$ belong to $\C0\left([0,\infty);\Lloc1(\reali^n)\right)$,
we conclude that
\[
\int_{\Omega_\tau}\modulo{u(\tau,x)-\bar u(\tau,x)}\d x
\leq
\int_{\Omega_{\tau_0}}\modulo{u(\tau_0,x)-\bar u(\tau_0,x)}\d x.
\]
Finally, if $\tau_0=0$, we obtain~\eqref{eq:main} by continuity of $u=u(t)$ and $\bar u=\bar u(t)$.
\end{proof}

\section{Corollaries}

\subsection{Time-dependent bounds}

Theorem~\ref{thm:main} can be slightly generalized in the following way. Suppose that $u$ and $\bar u$ lie between two continuous functions $a,b\colon[0,\infty)\to\reali$. 
Then the theorem still holds if we choose  
\[
F(t,x) = \co f\left(t,x,\left[a(t),b(t)\right]\right).
\]
To see this, one can use a simple convergence argument. For each $h\in(0,\infty)$, define a set-valued map $U_h$ by the rule 
\[
  U_h(t) = \left[a_h^k, b_h^k\right],
  \qquad t\in \big[kh, (k+1)h\big),\qquad k \in \mathbb{Z} _+,
\]
where 
\[
  a_h^k = \min\left\{a(t)\;\colon\;t\in [kh,(k+1)h]\right\}, \qquad
  b_h^k = \max\left\{b(t)\;\colon\;t\in [kh, (k+1)h]\right\}.
\]

Fix a time interval $[\tau_0,\tau]$ and denote by $\Omega^-(K)$ and $\Omega^{h-}(K)$ 
the \emph{backward} integral funnels of inclusions
\[
  \dot x\in \co f\left(t,x, \left[a(t),b(t)\right]\right)\qquad\mbox{and}\qquad
  \dot x\in \co f\left(t,x,U_h(t)\right).
\]
Lemma~\ref{lem:Omega_conv} from Appendix says that the slices 
$\Omega^-_{\tau _0}(K)$ and $\Omega^{h-}_{\tau _0}(K)$ 
converge to each other in the sense that
\[
  \mathcal{L}^n\left(\Omega^{h-}_{\tau _0}(K)\triangle\Omega^-_{\tau _0}(K)\right)\to 0
  \qquad\text{as}\qquad h\to 0.
\]

On the other hand, Theorem~\ref{thm:main} implies the inequality
\[
\int_{K}\modulo{u(\tau,x)-\bar u(\tau,x)}\d x
\leq
\int_{\Omega^{h-}_{\tau_0}(K)}\modulo{u(\tau_0,x)-\bar u(\tau_0,x)}\d x.
\]
We can rewrite it as
\begin{multline*}
\int_{K}\modulo{u(\tau,x)-\bar u(\tau,x)}\d x
\leq
\int_{\Omega^-_{\tau_0}(K)}\modulo{u(\tau_0,x)-\bar u(\tau_0,x)}\d x\\ 
+\bigg(
  \int_{\Omega^{h-}_{\tau_0}(K)}\modulo{u(\tau_0,x)-\bar u(\tau_0,x)}\d x 
-
\int_{\Omega^-_{\tau_0}(K)}\modulo{u(\tau_0,x)-\bar u(\tau_0,x)}\d x
\bigg).
\end{multline*}
The expression in parentheses can be estimated from above by the quantity
\[
  C\mathcal{L} ^n \left(\Omega ^{h-}_{\tau_0} (K)\triangle\Omega^-_{\tau_0}(K) \right),
\]
for a certain positive $C$. 
Hence, by passing to the limit in the previous inequality, we obtain the desired result:
\[
\int_{K}\modulo{u(\tau,x)-\bar u(\tau,x)}\d x
\leq
\int_{\Omega^-_{\tau_0}(K)}\modulo{u(\tau_0,x)-\bar u(\tau_0,x)}\d x.
\]

\subsection{Estimation of the domain of dependence}

To estimate the domain of dependence we need two extra assumptions:
\begin{enumerate}
\item[$(\mathbf{f_{3}})$] the function $f$ is smooth, both maps $\partial_u f$ and $(t,x)\mapsto \div\left(f(t,x,0)\right)$ 
  are bounded\footnote{One can choose other assumptions (for example, those from \S4 of~\cite{Kruzkov}) that guarantee the applicability of
the vanishing viscosity method.}.
\item[$(\mathbf{f_{4}})$] $f(t,x,0) = 0$ for all $t$ and $x$.
\end{enumerate}

\begin{lemma} 
\label{lem:pointwise}
Let assumptions $(\mathbf{f_{1}})$--$(\mathbf{f_{4}})$ hold. Suppose that $u$ is
a Kru\v{z}kov solution of~\eqref{eq:genpde}, \eqref{eq:intcon}. If $u_0(x)\in
[a_0,b_0]$ for all $x\in \mathbb{R}^n$, then $u(t,x)\in [a(t),b(t)]$ for almost
all $(t,x)\in [0,\infty)\times \mathbb{R}^n$, where $a$, $b$ are defined by
\begin{equation}
\label{eq:ab}
a(t) = 
\begin{cases}
  a_0 e^{-nL_1t}, & a_0\geq 0,\\
  a_0 e^{nL_1t}, & a_0 < 0,
\end{cases}
\qquad
b(t) = 
\begin{cases}
  b_0 e^{nL_1t}, & b_0\geq 0,\\
  b_0 e^{-nL_1t}, & b_0 < 0.
\end{cases}
\end{equation}
\end{lemma}
\begin{proof}
\textbf{1.} For an arbitrary $\epsilon>0$, consider the following parabolic equation
\begin{equation}
\label{eq:parabolic}
\partial_t u + \div \left(f(t,x,u)\right) = \epsilon\Delta u.
\end{equation}
Due to assumptions $(\mathbf{f_{1}})$--$(\mathbf{f_{4}})$, we may write
\begin{equation}
\label{eq:meanvalue}
f(t,x,w) = f(t,x,0) + \partial_u f\left(t,x,\tilde u(t,x,w)\right) w
= \partial_u f\left(t,x,\tilde u(t,x,w)\right) w,
\end{equation}
for some smooth $\tilde u$. Therefore, any classical solution $u^\epsilon$ of~\eqref{eq:parabolic} 
also solves the linear equation
\[
\partial_t u + \div \left(c(t,x)u\right) = \epsilon\Delta u,
\]
where $c$ is defined by
\[
  c(t,x) = \partial _u f\left(t,x,\tilde u\left(t,x,u^{\epsilon}(t,x)\right)\right).
\]
The latter PDE may be written in the form $\mathcal{L}[u]=0$ after 
introducing the parabolic operator
\[
\mathcal{L}[u] = \epsilon\Delta u - \div \left(c(t,x)u\right) - \partial_t u.
\]

\textbf{2.} Suppose that $u_0$ is a smooth function. Then the Cauchy problem for~\eqref{eq:parabolic} with
$u_0$ taken as initial condition admits a classical solution~\cite{Ladyzhenskaya}, which we denote by $u^{\epsilon}$.

To get the upper bound, we must consider two cases depending on the sign if $b_0$. For $b_0\geq 0$,
we take $Z(t,x) = b_0 e^{nL_1 t}$ and note that
\[
  \mathcal{L}[Z] = - b_0 e^{nL_1 t}\left(\div c(t,x) + n L_1\right) \leq 0,
\]
because of $(\mathbf{f_{2}})$ (more precisely, $(\mathbf{g_{2}})$ for $g=\partial_uf$). 
Since $u_0(x)\leq Z(0,x)$, the maximum principle for parabolic 
equations~\cite[Theorem 12, p. 187]{ProtterBook} implies that
$u^{\epsilon}(t,x)\leq b_0 e^{nL_1 t}$.
For $b_0<0$, we choose $Z(t,x)=b_0e^{-nL_1t}$ and get
$u^{\epsilon}(t,x)\leq b_0 e^{-nL_1 t}$, by the same arguments.
The lower bound can be obtained similarly.

By the method of vanishing viscosity~\cite[Theorem 4]{Kruzkov}, we have 
$u(t,x) = \lim_{\epsilon\to 0}u^{\epsilon}(t,x)$ for a.e. $t$ and $x$. 
This proves the lemma for smooth initial data.

\textbf{3.} To deal with an arbitrary $u_0$ we may use the usual mollification technique (see the arguments right above Theorem 4 in~\cite{Kruzkov}).
\end{proof}

Recall that the collection $\mathcal{K}(\reali^n)$ of all nonempty compact subsets of $\reali^n$
is a metric space with the Hausdorff distance 
\[
d_H(A,A') = \inf\{r>0\;\colon\;A\subseteq B(A',r),\; A'\subseteq B(A,r)\}.
\]
When we say that a set-valued map with compact values is continuous we always mean
the continuity with respect to the Hausdorff distance.

Let us take $\epsilon>0$ and perturb $a(\cdot)$, $b(\cdot)$ 
from Lemma~\ref{lem:pointwise} as follows
\begin{equation}
  \label{eq:abeps}
  a_\epsilon(t) = 
\begin{cases}
  (a_0-\epsilon) e^{-nL_1t}, & a_0\geq 0,\\
  (a_0-\epsilon) e^{nL_1t}, & a_0 < 0,
\end{cases}
\qquad
b_\epsilon(t) = 
\begin{cases}
  (b_0+\epsilon) e^{nL_1t}, & b_0\geq 0,\\
  (b_0+\epsilon) e^{-nL_1t}, & b_0 < 0.
\end{cases}
\end{equation}
Consider the set-valued map
\begin{equation}
\label{eq:Feps}
F_{\epsilon}(s,y) = \co\partial_u f\left(s,y,\left[a_\epsilon(s),b_\epsilon(s)\right]\right)
\end{equation}
and denote by $\Omega^{\epsilon}\left(B(x,\epsilon)\right)$ the \emph{backward} integral funnel of the differential inclusion
\[
\begin{cases}
\dot y \in F_{\epsilon}(s,y), \qquad s\in [0,t],\\
y(t)\in B(x,\epsilon).
\end{cases}
\]

\begin{lemma}
\label{lem:convergence}
The set-valued map $\epsilon\mapsto \Omega^{\epsilon}_{0}\left(B(x,\epsilon)\right)$
is continuous.
\end{lemma}
\begin{proof}
  The map $U\colon (s,\epsilon)\mapsto \left[a_\epsilon(s),b_\epsilon(s)\right]$ is obviously continuous. Therefore,
  the composition $(s,y,\epsilon)\mapsto \partial_u f\left(s,y,U(s,\epsilon)\right)$ is continuous as well~\cite[Proposition 2.56]{Papageorgiou}. Since convexification preserves continuity 
  (see~\cite[Proposition 2.42]{Papageorgiou}), also the map  
  $(s,y,\epsilon)\mapsto F_{\epsilon}(s,y)$ is continuous. It remains to apply Theorem 5.4 from~\cite[p. 213]{Tolstonogov} to complete the proof.
\end{proof}

Now we have everything at hand to estimate the domain of dependence.

\begin{corollary}
\label{cor:domestim}
Suppose that assumptions $(\mathbf{f_{1}})$--$(\mathbf{f_{4}})$ hold.
Let $u$ be a Kru\v{z}kov solution of~\eqref{eq:genpde}, \eqref{eq:intcon} and 
the values of $u_0$ belong to an interval $[a_0,b_0]$. Then, for each $x\in \reali^n$, we have
\[
\mathcal{D}_u(t,x)\subseteq \Omega^-_{0}(x).
\]
Here $\Omega^-(x)$ is the backward integral funnel of the differential inclusion
\[
  \dot y(s)\in \co\partial_u f\left(s,y(s),\left[a(s),b(s)\right]\right),\qquad y(t)=x,
\]
where $a(\cdot)$ and $b(\cdot)$ are defined by~\eqref{eq:ab}.
\end{corollary}
\begin{proof}
  Let $v_0=u_0+\epsilon w$ as in Definition~\ref{def:dd}. Without loss of
  generality, we may assume that $\|w\|_{\L\infty(\reali^n)}=1$. In this case,
  $v_0$ takes values in $[a_0-\epsilon, b_0+\epsilon]$. Hence, by
  Lemma~\ref{lem:pointwise}, the Kru\v{z}kov solution $v$ of~\eqref{eq:genpde},
  \eqref{eq:intcon} with initial data $v_0$ can be estimated as follows
  \[
    a_{\epsilon}(t) \leq v(t,x) \leq b_{\epsilon}(t)\qquad\text{for all}\quad 
    t\in [0,\infty),\;\; x\in\reali^n,
  \]
  where $a_{\epsilon}(\cdot)$ and $b_{\epsilon}(\cdot)$ are given by~\eqref{eq:abeps}.

  Suppose that $\spt w\cap\Omega^-_{0}(x)=\varnothing$. Since $\reali^n$ is a normal space, 
  there exists an open neighbourhood $\mathcal{O}\left(\Omega^-_{0}(x)\right)$ of 
  the set $\Omega^-_{0}(x)$ such that
  \[
  \spt w\cap \mathcal{O}\left(\Omega^-_{0}(x)\right) = \varnothing.
  \]
  By Lemma~\ref{lem:convergence}, $\Omega^{\epsilon}_{0}\left(B(x,\epsilon)\right)$
  belongs to this neighbourhood, for all $\epsilon$ small enough. Hence,
  \[
  \spt w\cap \Omega^{\epsilon}_{0}\left(B(x,\epsilon)\right) = \varnothing.
  \]

  Now, recalling that Theorem~\ref{thm:main} holds also for time-dependent bounds, we may write
  \[
  \int_{B(x,\epsilon)} |u(\tau,y)-v(\tau,y)|\d y \leq 
  \epsilon\int_{\Omega^{\epsilon}_{0}\left(B(x,\epsilon)\right)}w(y)\d y.
  \]
  Previously, we have proved that the right-hand side vanishes for small $\epsilon$. 
  Hence, after dividing both 
  parts of the inequality by $\mathcal{L}^n\left(B(x,\epsilon)\right)$ and passing to the limit
  as $\epsilon\to 0$, we discover that $u$ and $v$ coincide at $(t,x)$ in the sense of
  equation~\eqref{eq:conicide}. The proof is complete.
\end{proof}

\subsection{Expansion of the support}

The following result shows how differential inclusions can be used to estimate 
the support of Kru\v{z}kov solutions.

\begin{corollary}
\label{cor:support}
Suppose that assumptions $(\mathbf{f_{1}})$--$(\mathbf{f_{4}})$ hold.
Let $u$ be a Kru\v{z}kov solution of~\eqref{eq:genpde}, \eqref{eq:intcon} such that
\[
    0\leq a(t) \leq u(t,x) \leq b(t)\qquad\text{for all}\quad 
    t\in [0,\infty),\;\; x\in\reali^n,
\]
where $a(\cdot)$ and $b(\cdot)$ are continuous functions.
If $u_0$ has compact support, then 
\[
\spt u(t,\cdot)\subseteq \Omega^+_t(\spt u_0),\qquad t\in [0,\infty),
\]
where $\Omega^+(\spt u_0)$ is the forward integral funnel of the differential inclusion
\[
  \dot y(s)\in \co\partial_u f\left(s,y(s),\left[a(s),b(s)\right]\right),\qquad y(0)\in \spt u_0.
\]

\end{corollary}
\begin{proof}
Let $K=\spt u_0$. Construct the forward funnel $\Omega^+(K)$
issuing from $K$ at the time moment $t=0$. Fix a positive $t$, choose a closed ball $B$ outside $\Omega_{t}^+(K)$ and construct the 
backward funnel $\Omega^-\left(B\right)$ issuing from $B$ at the time moment $t$.
Clearly, $\Omega^+(K)$ and $\Omega^-\left(B\right)$ cannot intersect. Therefore, $u_0=0$ a.e. on $\Omega_{0}^-\left(B\right)$.
According to $(\mathbf{f_{4}})$, we may take $\bar u\equiv 0$ in~\eqref{eq:main} to obtain 
\[
\int_{B}u(t,x)\d x
\leq
\int_{\Omega^-_{0}\left(B\right)}u_0(x)\d x,
\]
which means that $u(t,\cdot)=0$ a.e. on $B$.

Since the set $\left(\Omega_{t}^+(K)\right)^c$ is open, there exists a countable collection
of closed balls $B_i$ from $\left(\Omega_{t}^+(K)\right)^c$ such that
\[
\mathcal{H}^{n}\left(\left(\Omega_{t}^+(K)\right)^c\setminus \bigcup_i B_i\right)=0
\]
(see~\cite[Corollary 2, \S 1.5]{EvansGariepy}). By the above arguments $u(t,\cdot)=0$ a.e. on every $B_i$. The proof is complete.
\end{proof}

\section{Application}

Consider the following nonlinear conservation law
\begin{equation}
\label{eq:controlled}
\partial u_t + \div\left(f\left(t,x,u,\xi(t)\right)\right) = 0,
\end{equation}
where $\xi=\xi(t)$ is a control parameter. Suppose that the initial
function $u_0$ is nonnegative and $f$ takes the form
\begin{equation}
\label{eq:conrolf}
f(t,x,u,\xi) = v(t,x,\xi)u + G(u),
\end{equation}
where 
\begin{equation}
\label{eq:Gassumptions}
G(0) = 0\quad\mbox{and}\quad |G'(u)|\leq c\quad \mbox{for some}\quad c>0. 
\end{equation}

By Theorem 3 in~\cite{Kruzkov}, the nonnegativity of $u_0$ implies the nonnegativity of the corresponding 
Kru\v{z}kov solution $u$. Therefore, we may think of a mass distributed on $\reali^n$ and drifting along 
the flow $f$; in this case $u(t,\cdot)$ represents the density at the time moment $t$.

Consider the following confinement problem: given compact set $K\subset \reali^n$ 
and an initial mass distribution $u_0$
with $\spt u_0\subseteq K$, find a control strategy $\xi=\xi(t)$ so that $\spt u(t,\cdot)\subseteq K$ for all $t>0$.

In other words, we want to keep the mass inside $K$ for every $t>0$, by choosing a suitable $\xi=\xi(t)$.

Let $u_\xi$ be the Kru\v{z}kov solution of~\eqref{eq:controlled} satisfying the initial function $u_0$ and the control strategy $\xi=\xi(t)$.
Note that 
\[
\partial_u f\left(t,x,[0,\infty),\xi\right) \subseteq v(t,x,\xi) + B(0,c).
\]
Hence the support of $u_\xi(t,\cdot)$ may be estimated by the reachable set of the differential inclusion
\[
\dot x \in v\left(t,x,\xi(t)\right) + B(0,c).
\]
We denote this reachable set by $\Omega^+_t(\spt u_0,\xi)$ indicating that it depends also on the control $\xi=\xi(t)$.
To be more precise, by Corollary~\ref{cor:support}, we may write
\[
\spt u_\xi(t,\cdot)\subseteq \Omega^+_t(\spt u_0,\xi),\qquad t\in [0,\infty).
\]
Thus, any function $\xi$ satisfying
\[
\Omega^+_t(\spt u_0,\xi) \subseteq K,\qquad t\in [0,\infty),
\]
solves our initial confinement problem.

Such $\xi$ may be constructed by applying, for example, the technique developed in~\cite{ColomboPogodaev1, ColomboPogodaev2}.
In particular, one may prove the following result.

\begin{theorem}
  \label{thm:Confinen}
  Let $n \in \naturali$ with $n \geq 2$ and the map $f\colon
[0,\infty)\,\times\,\reali^n\times\reali\times\reali^n\to\reali^n$ take the
form~\eqref{eq:conrolf}, where $G$ satisfies~\eqref{eq:Gassumptions} and $v$ is
given by
  \[
  v(t,x,\xi) = \psi\left(|x-\xi|\right)\,(x-\xi),
  \]
  for some real-valued function $\psi$. Let the map
  $(t,x,u)\mapsto f\left(t,x,u,\xi(t)\right)$ satisfies assumptions
$(\mathbf{f_{1}})$--$(\mathbf{f_{4}})$, for any smooth $\xi = \xi(t)$. Denote by
$u_\xi$ the Kru\v{z}kov solution of~\eqref{eq:controlled} corresponding to the
initial function $u_0$ and the control strategy $\xi$. Suppose that there exist
positive $R_*^-$, $R_*^+$ and $R$ such that
  \begin{equation*}
    \frac{1}{\sqrt{\pi}}\cdot
    \frac{\Gamma\left(\frac{n}{2}\right)}
    {\Gamma\left(\frac{n-1}{2}\right)}
    \int_0^\pi
    \psi \left(\sqrt{R^2 + R_*^2 - 2 R_* \, R\, \cos\theta} \right)
    (R_* - R \, \cos\theta) \, \sin^{n-2} \theta \, \d\theta
    <
    -c
  \end{equation*}
  for all $R_* \in [R_*^-,\, R_*^+]$.  Then, there exists a smooth control $\bar\xi \colon [0,\infty)\,\to \partial B(0,R)$ such that
  \begin{equation*}
    \spt u_0 \subseteq B (0, R_*^-)
    \quad \mbox{ implies } \quad
    \spt u_{\bar\xi}(t,\cdot) \subseteq B (0, R_*^+)
    \quad \mbox{ for all } \quad t\in [0,\infty) \,.
  \end{equation*}
\end{theorem}

Remark that the proof of the above theorem is constructive, in the sense that the confining strategy $\bar\xi$ is explicitly defined.


\section{Appendix: proofs of auxiliary results}

Before passing to the proofs, let us introduce a few notions from Geometric Measure Theory.

Consider a compact subset $E$ of $\reali^n$. The \emph{$n$-dimensional density} of $E$ at $x$ is given by
\[
\Theta^n(E,x) = \lim_{r\to 0+} \frac{\mathcal{H}^n\left(E\cap B(x,r)\right)}{\mathcal{H}^n(B(x,r))},
\]
the symbol $E^t$ denotes the set of all points where $E$ has density $t$. 
The \emph{essential boundary} of $E$ is the set of points where $E$ has density strictly between $0$ and $1$, i.e.,
$\partial^*E = \reali^n\setminus (E^0\cup E^1)$. The \emph{perimeter} of $E$ is
denoted by $P(E)$.
Here we intentionally skip the exact definition of the perimeter
(one may find it, for instance, in~\cite{AFP}), instead we use the formula $P(E) = \mathcal{H}^{n-1}(\partial^* E)$,
which holds for any compact set $E$ (see~\cite[p. 159]{AFP}).

\begin{proofof}{Proposition~\ref{prop:tubular}}
The $(n-1)$-rectifiability of $\partial A$ was established in~\cite[Proposition 2.3]{Rataj}. 
The very definition of $(n-1)$-rectifiability implies that $\mathcal{H}^{n-1}(\partial A)<\infty$.
Thus, we only need to show that the outer Minkovski content exists and it is bounded by $\mathcal{H}^{n-1}(\partial A)$.

If $x\in \partial A$, then for all sufficiently small $r>0$, one may find an $n$-dimensional closed ball $B_{r/2}$ of radius $r/2$
such that $x\in B_{r/2}$ and $B_{r/2} \subset A$. Using the obvious inclusion
$B_{r/2} \subseteq A\cap B(x,r)$,
we deduce that
\[
\Theta^{n}\left(A,x\right)\geq 2^{-n}.
\]
It follows from the arbitrariness of $x$ that $\partial A\cap A^0 = \varnothing$.

Now, taking into account the rectifiability of $\partial A$, we may apply Proposition 4.1 from~\cite{Villa},
which states that
\[
\mathcal{SM}^{n}\left(A\right) = P\left(A\right) + 2\mathcal{H}^n\left(\partial A\cap A^0\right).
\]
Since the second term from the right-hand side is zero, we get
$\mathcal{SM}^{n}\left(A\right) = P\left(A\right)$. The perimeter is bounded,
because $P(A) = \mathcal{H}^{n-1}(\partial^* A)\leq \mathcal{H}^{n-1}(\partial
A)$.
\end{proofof}

We have said before that integral funnels are ``almost'' like tubular neighbourhoods. The next theorem, taken from~\cite{Pogodaev2016},
clarifies this statement.

\begin{theorem}
\label{thm:intball}
Suppose that $K$ is a compact $n$-dimensional tubular neighbourhood, $F$ is defined 
by~\eqref{eq:F}, and assumptions
$(\mathbf{g_{1}})$--$(\mathbf{g_{3}})$ hold. Then
\[
\Omega^+(K)\cup \{(t,x)\;\colon\; t\leq t_0,\; t\geq t_1\}
\]
is an $(n+1)$-dimensional tubular neighbourhood.
\end{theorem}

We derive Proposition~\ref{prop:rectifiable} from that theorem.

\begin{lemma}
\label{lem:boundaries}
Essential and topological boundaries of $\Omega^+(K)$ coincide. 
\end{lemma}
\begin{proof}
Below, for the sake of brevity, we write $\Omega$ and $\Omega_t$ instead of $\Omega^+(K)$ and $\Omega^+_t(K)$.

Let us fix a point $(t,x)\in\partial\Omega$ and estimate the density of $\Omega$ at $(t,x)$.

\textbf{1.} Suppose that  $\tau_0<t<\tau$. As in the proof of Proposition~\ref{prop:tubular}, we find that
\[
\Theta^{n+1}\left(\Omega,(t,x)\right)\geq 2^{-(n+1)}.
\]
On the other hand, our assumptions yield that the set-valued map $t\mapsto \Omega_t$ is Lipschitz.
In other words,
\[
  \Omega_s \subseteq B\left(\Omega_t, L|s-t|\right)\qquad\text{for all}\quad  s\in [\tau_0,\tau],
\]
or, equivalently, 
\[
\Omega\subseteq C = \left\{(s,y)\;\colon\;y\in B\left(\Omega_t, L|s-t|\right)\right\}.
\]
The latter implies that $\Theta^{n+1}\left(\Omega,(t,x)\right)\leq \Theta^{n+1}\left(C,(t,x)\right)<1$.

\textbf{2.} If $t$ is either $\tau_0$ or $\tau$, we obviously get that $\Theta^{n+1}\left(\Omega,(t,x)\right)\leq 2^{-1}$.
It remains to establish the inequality $\Theta^{n+1}\left(\Omega,(t,x)\right) > 0$.

We begin with the case $t=\tau_0$. Consider the $(n+1)$-dimensional ball $B\big((\tau_0,x),r\big)$ with a sufficiently small radius $r$. Since $K$ is a tubular neighbourhood, there exists an $n$-dimensional ball $A\subseteq K$ of radius $r/2$ with $x\in \partial A$.

Take a map $u\in \L\infty([\tau_0,\tau];U)$ and denote by $P_{s,t}$ the phase flow of the vector field $(t,x)\mapsto g\left(t,x,u(t)\right)$.
Clearly, we have $P_{\tau_0,t}(A)\in \Omega_t$, for all $t\in [\tau_0,\tau]$. Since $g$ is continuous according to $(\mathbf{g_1})$,
its norm is bounded on $\Omega\times U$ by some positive $c$. Therefore, the set
\[
\Pi = \left\{(t,y)\;\colon\;t\in [\tau_0,t_*],\; y\in P_{\tau_0,t}(A)\right\}
\]
is surely contained in $\Omega\cap B\left((\tau_0,x),r\right)$ when $t_*=\tau_0+\tfrac{r}{2\sqrt{1+c^2}}$.

Let us estimate $\mathcal{H}^{n+1}(\Pi)$. Using the notation $V(t) = \mathcal{H}^n\left(P_{\tau_0,t}(A)\right)$, we can write
\[
\mathcal{H}^{n+1}(\Pi) = \int_{\tau_0}^{t_*} V(t)\d t.
\]
According to the Reynolds transport theorem~\cite{LorenzRTT}, $V$ is absolutely continuous and
\begin{equation}
\label{eq:RTT}
\frac{d}{dt} V(t) = \int_{P_{\tau_0,t}(A)}\div g\left(t,y,u(t)\right)\d y\quad\text{for a.e.}\quad  t\in [\tau_0,\tau].
\end{equation}
It follows from $(\mathbf{g_2})$ that $\div g\geq -n L_1$, which together with~\eqref{eq:RTT} implies the inequality $V(t)\geq V(\tau_0)e^{-nL_1(t- \tau_0)}$. 
Using the latter estimate, we find that
\[
\mathcal{H}^{n+1}(\Pi) \geq \int_{\tau_0}^{t_*} V(\tau_0)e^{-nL_1(t- \tau_0)} \d t = \frac{V(\tau_0)}{nL_1}
\left(1 - e^{-\frac{nL_1}{2\sqrt{1+c^2}}r}\right).
\] 
Dividing by $\omega_{n+1}r^{n+1}$ and passing to the limit as $r\to 0$, we obtain
\[
\Theta^{n+1}\left(\Omega,(\tau_0,x)\right)\geq \frac{\omega_n}{nL_1 \omega_{n+1}}\lim_{r\to 0}
\frac{1}{r}
\left(1 - e^{-\frac{nL_1}{2\sqrt{1+c^2}}r}\right) = \frac{\omega_{n}}{2\sqrt{1+c^2}\omega_{n+1}},
\] 
which is strictly greater than zero, as desired.

The case $t=\tau$ can be considered in the similar way. Of course, $P_{\tau_0,t}$ should be replaced with $P_{t,\tau}$ and $K$ with $\Omega_\tau(K)$, which is also a tubular neighbourhood
by~\cite[Theorem 2.1]{Lorenz05}.
\end{proof}

\begin{proofof}{Proposition~\ref{prop:rectifiable}}
As before, we write $\Omega$ and $\Omega_t$ instead of $\Omega^+(K)$ and $\Omega^+_t(K)$.

\begin{wrapfigure}{hR}{0.50\textwidth}
\centering
\includegraphics[width=0.47\textwidth]{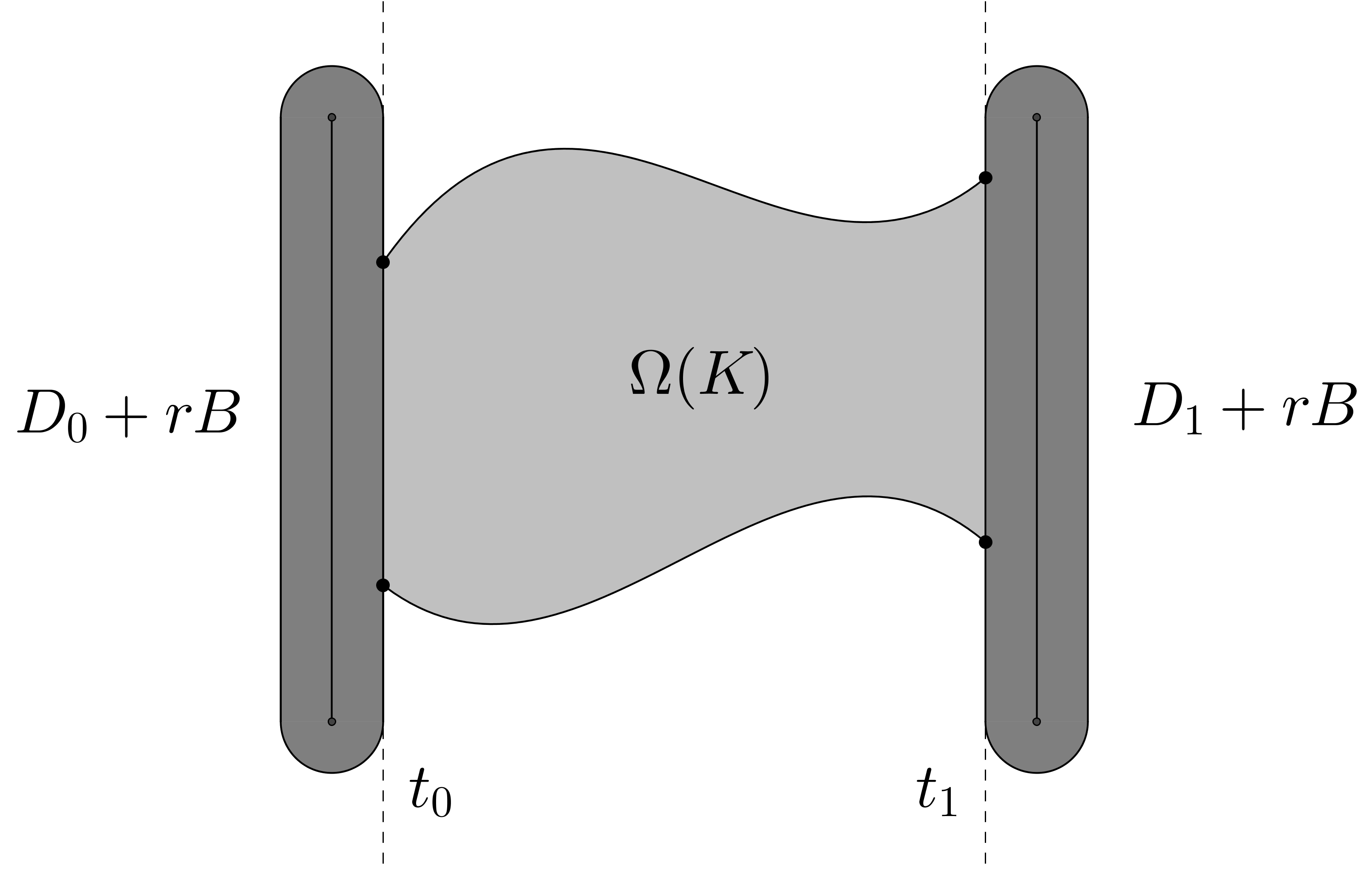}
\caption{\label{fig:ibp} The funnel with two enlarged discs.\vspace{10pt}}
\end{wrapfigure}

\textbf{1.} Let us show that $\partial\Omega$ is $n$-rectifiable.
First, note that Theorem~\ref{thm:intball} implies that $\Omega$ is a part of a \emph{compact} $(n+1)$-dimensional 
tubular neighbourhood. Indeed, take two $n$-dimensional discs $D_0$ and $D_1$ lying in $\reali^{n+1}$,
enlarge them to get $B(D_0,r)$ and $B(D_1,r)$, and attach these sets to $\Omega^+(K)$ as is shown on Figure~\ref{fig:ibp}.
The resulting set 
\[
A=\Omega\cup B(D_0,r)\cup B(D_1,r)
\]
is a compact $(n+1)$-dimensional tubular neighbourhood, as desired. 

Now, Proposition~\ref{prop:tubular} implies that $\partial A$ is $n$-rectifiable. One may easily check that a finite union of $n$-rectifiable sets 
and a subset of any $n$-rectifiable set are $n$-rectifiable. These facts allow us to derive the $n$-rectifiability of $\partial\Omega$.

\textbf{2.} Taking into account the rectifiability of $\partial \Omega$, we may apply Proposition 4.1 from~\cite{Villa}
to get
\[
\mathcal{SM}^{n+1}\left(\Omega\right) = P\left(\Omega\right) + 2\mathcal{H}^n\left(\partial \Omega\cap \Omega^0\right).
\]
By Lemma~\ref{lem:boundaries}, the last term from the right-hand side is zero
and $P(\Omega) = \mathcal{H}^{n}(\partial\Omega)$. Therefore,
$\mathcal{SM}^{n}(\Omega)=\mathcal{H}^{n}(\partial \Omega)$, which is finite due
to $n$-rectifiability of $\partial \Omega$.

\textbf{3.} It remains to note that part (b) of the statement follows from~\cite[Theorem 2.1]{Lorenz05}.
\end{proofof}

Next, we put here a small auxiliary lemma, which directly follows from Proposition 4.1 of~\cite{Villa}.
\begin{lemma}
\label{lem:content} Let $A\subset \reali^n$ be $(n-2)$-rectifiable. Then $\mathcal{SM}^n(A) = 0$.
\end{lemma}
\begin{proof}
Since $A$ is $(n-2)$-rectifiable, it is $(n-1)$-rectifiable as well. Thus, we may apply
Proposition 4.1~\cite{Villa}, which gives
\[
\mathcal{SM}^{n}(A) = P(A) + 2\mathcal{H}^{n-1}(\partial A\cap A^0)\leq 3 \mathcal{H}^{n-1}(\partial A) = 0.
\]
\end{proof}

We conclude the section with another technical lemma. Suppose that $h\in
(0,\infty)$ and $U_h,U\colon [\tau_0,\tau]\to \mathcal{K}(\reali^l)$ are
measurable set-valued maps with the following properties: a) $U_h(t)\to U(t)$
for each $t\in [\tau_0,\tau]$ as $h\to 0$, b) $U_h$ and $U$ take values in a
bounded set $B\subset \mathbb{R}^l$, at least, for all sufficiently small $h$.
Let $\Omega^{h+}(K)$ and $\Omega^+(K)$ denote the forward integral funnels of
the following differential inclusions
\begin{align*}
  \dot x(t)\in \co g\left(t,x,U_h(t)\right), \qquad
  \dot x(t)\in \co g\left(t,x,U(t)\right),\qquad t\in [\tau _0, \tau ].
\end{align*}

\begin{lemma}
  \label{lem:Omega_conv}
  Let $K$ be an $n$-dimensional tubular neighbourhood. Then, under the assumptions 
  $\mathbf{(g_1)}$--$\mathbf{(g_3)}$, we have
  \begin{equation}\label{eq:Omega_conv}
    \lim_{h\to 0}\mathcal{L}^n\left(\Omega^{h+}_{t}(K)\triangle\Omega^+_{t}(K)\right) = 0,
  \end{equation}
  where $\triangle$ stands for the symmetric difference between two sets.
\end{lemma}
\begin{proof}
  \textbf{1.} Let us show that $\Omega^{h+}_{t}(K)$ and $\Omega^+_{t}(K)$ converge 
  in the Hausdorff distance. To this end, consider two trajectories $x(\cdot)$ and 
  $x_h(\cdot)$ defined by
  \begin{align*}
    \dot x (t) & = g\left(t,x(t),u(t)\right), &x(\tau _0) &= x_0, 
                  &u(t)&\in U(t),\\
    \dot x_h (t) & = g\left(t,x(t),u_h(t)\right), &x_h(\tau _0) &= x_0, 
                  &u_h(t)&\in U_h(t),
  \end{align*}
  where $u_h(t)$ is the projection of $u(t)$ onto $U_h(t)$. Obviously, we may write
  \begin{align*}
    \left|x(t)-x_h(t)\right| &\leq 
    \int_0^t 
    \left|g\left(s,x(s),u(s)\right)-g\left(s,x_h(s),u_h(s)\right)\right| 
    \d s \\
    &\leq
    \int_0^t 
    \left|g\left(s,x(s),u_h(s)\right)-g\left(s,x_h(s),u_h(s)\right)\right| 
    \d s \\ 
    &+
    \int_0^t 
    \left|g\left(s,x(s),u(s)\right)-g\left(s,x(s),u_h(s)\right)\right| 
    \d s.
  \end{align*}
  According to $\mathbf{(g_2)}$, the first integral from the right-hand side can be estimated
  by 
  \[
    \int_0^t L_1\left|x(s)-x_h(s)\right|\d s. 
  \]
  Hence, denoting the second integral by $c_h(t)$, 
  we obtain
  \begin{equation*}
     \left|x(t)-x_h(t)\right| \leq \int_0^t L_1\left|x(t)-x_h(t)\right|\d s + c_h(t).
  \end{equation*}
  Gronwall's lemma yields that
  \begin{equation}
    \label{eq:lem_gron}
    \left|x(t)-x_h(t)\right|\leq c_h(t) e^{L_1t}.
  \end{equation}
  Since $|u(t)-u_h(t)|\leq d_H\left(U(t),U_h(t)\right)$ for each $t\in
  [\tau_0,\tau]$, we conclude that $u_h$ pointwise converges to $u$. Now
  assumptions $\mathbf{(g_1)}$ and $\mathbf{(g_2)}$ together with Lebesgue's
  dominated convergence theorem imply that $c_h(t)\to 0$ for each $t\in
  [\tau_0,\tau]$. Hence $x_h(t)\to x(t)$ for all $t\in [\tau_0,\tau]$,
  by~\eqref{eq:lem_gron}. Now, from Filippov's lemma~\cite{Filippov62} and the
  density theorem~\cite[Theorem 2.1., Chapter 3]{Tolstonogov} it follows that
  $d_H\left(\Omega^{h+}_t(K),\Omega^+_t(K)\right)\to 0$ for each $t\in
  [\tau_0,\tau]$.

  \textbf{2.} Let us prove that, for any two tubular neighbourhoods $A$, $A'$ of radius $r$, we have
  \begin{equation}\label{eq:LebHaus}
    \mathcal{L}^n(A\triangle A') \leq
    n \omega _n\frac{(\diam A)^n+(\diam A')^n}{2^n}\ln\left(1+\frac{d_H(A,A')}{r}\right).
  \end{equation}
  Since $\mathcal{L}^n(A\triangle A') = \mathcal{L}^n(A\setminus
  A')+\mathcal{L}^n(A'\setminus A)$, we proceed examining each term from the
  right-hand side separately. First, note that $A'\subseteq B(A,l)$, where $l =
  d_H(A,A')$. Thus,
  \[
    \mathcal{L}^n(A'\setminus A)\leq \mathcal{L}^n \left( B(A,l) \right)- \mathcal{L}^n(A).
  \]
  The Reynolds transport theorem~\cite{LorenzRTT} yields
 \[
    \mathcal{L}^n \left( B(A,l) \right)- \mathcal{L}^n(A) = 
    \int_0^l \mathcal{H}^{n-1}\left(\partial B(A,t)\right)\d t.
 \]
 Lemma 2.5~\cite{alvarez_cardaliaguet_monneau_05} gives the following estimate
 \[
   \mathcal{H}^{n-1}\left(\partial B(A,t)\right)\leq \frac{n\, \omega _n (\diam A)^n}{2^n(t+r)}.
 \]
 Thus, 
 \[
    \mathcal{L}^n \left( B(A,l) \right)- \mathcal{L}^n(A) 
    \leq \frac{n\, \omega _n (\diam A)^n}{2^n}\ln\left(1+\frac{l}{r}\right).
 \]
 Interchanging the roles of $A$ and $A'$, we obtain the same inequality for
 $A'$, and therefore~\eqref{eq:LebHaus}.

 \textbf{3.} By Theorem 2.1~\cite{Lorenz05} the sets $\Omega _t^+(K)$ and 
 $\Omega _t^{h+}(K)$ 
 can be considered as tubular neighbourhoods of common radius $r$ that does not
 depend on $h$.  
 Moreover, because of assumptions $\mathbf{(g_1)}$ and $\mathbf{(g_2)}$, all 
 those sets lie in a ball of sufficiently large diameter.
 Thus, inequality~\eqref{eq:LebHaus} can be applied to all of them. Basically, it says that 
 the Hausdorff convergence of the reachable sets implies the convergence with respect
 to the Lebesgue measure. Now, to complete the proof we should only recall its first step.
\end{proof}

\subsection*{Acknowledgment}
The work was supported by the Russian Science Foundation, grant No 17-11-01093.

\small{

  \bibliography{laws}

  \bibliographystyle{abbrv}

}

\end{document}